\documentclass[11pt,a4paper,leqno]{article}

\usepackage{amssymb,amsfonts,amsmath,amsthm,mathrsfs,pdfpages}
\usepackage{graphicx}
\usepackage{xfrac}

\allowdisplaybreaks
\usepackage{color}
\usepackage[colorlinks=true,
citecolor=blue,
linkcolor=red,
anchorcolor=blue,
urlcolor=blue]{hyperref}

\numberwithin{equation}{section}

\newtheorem{Theorem}{Theorem}[section]
\newtheorem{Lemma}[Theorem]{Lemma}

\newtheorem{prop}[Theorem]{Proposition}

\def\D {D}

\def\H {{\mathcal H}}

\def\M {{\mathcal M}}
\def\N {{\mathcal N}}

\def\R {\mathbb{R}}
\def\Re {\mathfrak{Re\,}}

\def\eps{\varepsilon}
\def\e{{\rm e}}
\def\d{{\rm d}}

\def\i{{\rm i}}
\def \l {\langle}
\def \r {\rangle}
\def \and {{\qquad\text{and}\qquad}}
	
\title{\textbf{Exponential stability of Timoshenko-Gurtin-Pipkin systems with full thermal coupling}}

\author{
	\small {\bf Filippo Dell'Oro}\thanks{Corresponding author.
	 Email: \href{filippo.delloro@polimi.it}{ filippo.delloro@polimi.it.}}\\
	\small Politecnico di Milano, Dipartimento di Matematica, \\
	\small  Via Bonardi 9, 20133 Milano, Italy. \medskip
	\\
	\small {\bf Marcio A. Jorge Silva}\thanks{Supported by the CNPq, Grant \#301116/2019-9.}\\
	\small Department of Mathematics, State University of Londrina,\\
	\small  Londrina 86057-970, Paran\'a, Brazil. \medskip
	\\
	\small {\bf Sandro B. Pinheiro}\thanks{Supported by the CAPES, Finance Code 001.
	(Master and Ph.D. PICME Scholarships)
	}\\
	\small Department of Mathematics, State University of Maring\'a,\\
	\small Maring\'a 87020-900, Paran\'a,  Brazil.
}

\date{\vspace{-4ex}}

\begin{document}

\maketitle

\begin{abstract}
\noindent
We analyze the stability properties of a linear thermoelastic
Timoshenko-Gurtin-Pipkin system
with thermal coupling acting
on both the shear force and the bending moment. Under either the
mixed Dirichlet-Neumann or else the full Dirichlet boundary conditions,
we show that the associated solution semigroup
in the history space framework of Dafermos is
exponentially stable independently of the values of the structural parameters of the model.
\end{abstract}



\section{Introduction}\label{sec-intro}
\subsection{The model}

The vibrations of a Timoshenko beam of length $L>0$ are described by the linear evolution PDE system
\cite{timoshenko,timoshenko-1}
\begin{equation*}
\left\{\begin{array}{lcl}
\rho_1\varphi_{tt}-S_x =0,\smallskip\\
\rho_2\psi_{tt}-M_x+S =0,
\end{array} \right.
\end{equation*}
where $\varphi=\varphi(x,t)$ and $\psi=\psi(x,t)$ are functions
of the space-time variable $(x,t)\in (0,L) \times (0,\infty)$
and represent the vertical displacement
and the rotation angle of the cross-section of the beam, respectively.
The constants $\rho_{1},\rho_{2}>0$ are physical parameters of the model,
while $S$ and $M$ stand for the
shear force and the bending moment, respectively.
When the beam is subject to an unknown temperature distribution, one may assume that $S$ and $M$
satisfy the constitutive laws (see \cite{alvesetal-JEE})
\begin{equation*}
\left\{\begin{array}{lcl}
S =  k (\varphi_{x} +  \psi )  -   \gamma  \theta, \smallskip\\
M = b  \psi_{x}  -  \sigma\xi,
\end{array} \right.
\end{equation*}
where $\theta=\theta(x,t)$ and $\xi=\xi(x,t)$ represent the temperature
(deviations from a constant reference temperature)
along the longitudinal and vertical directions, respectively, and $k,b,\gamma,\sigma>0$ are further
physical parameters. To complete the picture, we need to consider
two additional equations describing the evolution
of $\theta$ and $\xi$. Here,
we employ the Gurtin-Pipkin thermal laws
\cite{gurtin-pipkin}
\begin{equation}\label{GP}
\left\{\begin{array}{lcl}
\displaystyle\rho_3 \theta_{t} -\varpi_1\int_{0}^{\infty} g(s) \theta_{xx}(t-s) \mathrm{d} s
+ \gamma (\varphi_{x} + \psi)_t = 0,\medskip\\
\displaystyle	\rho_4 \xi_{t} -
\varpi_2\int_{0}^{\infty} h(s) \xi_{xx}(t-s) \mathrm{d} s + \sigma \psi_{xt} = 0,
\end{array} \right.
\end{equation}
where $\rho_3,\rho_4,\varpi_1,\varpi_2>0$ are physical parameters and
the convolutions kernels $g,h:[0, \infty)\to[0, \infty)$ are convex integrable functions of unit total
mass, whose precise properties will be specified later.
The values of $\theta$ and $\xi$ for negative times are regarded as initial data of the problem.
Accordingly, we end up with the
following thermoelastic Timoshenko-Gurtin-Pipkin beam system with thermal coupling acting
on both the shear force and the bending moment
\begin{equation}\label{tim4-beg}
\left\{\begin{array}{lcl}
\rho_{1} \varphi_{tt} -  k (\varphi_{x} + \psi)_{x} + \gamma \theta_x = 0,\bigskip \\
\rho_{2} \psi_{tt} - b\psi_{xx} + k(\varphi_{x} + \psi) -\gamma \theta + \sigma \xi_x = 0,
\medskip\smallskip \\
\displaystyle	\rho_3 \theta_{t} -\varpi_1\int_{0}^{\infty} g(s) \theta_{xx}(t-s) \mathrm{d} s
+ \gamma (\varphi_{x} + \psi)_t= 0,
\medskip \\
\displaystyle	\rho_4 \xi_{t} -\varpi_2\int_{0}^{\infty} h(s) \xi_{xx}(t-s) \mathrm{d} s + \sigma \psi_{xt}= 0,
\end{array} \right.
\end{equation}
complemented with the initial conditions
\begin{equation}\label{init-data}
\left\{\begin{array}{lcl}
\varphi(x,0)= \varphi_0(x),\quad \varphi_t(x,0)= \Phi_0(x),\quad
\psi(x,0)= \psi_0(x),\quad \psi_t(x,0)= \Psi_0(x), \smallskip\\
\theta(x,0)=\theta_0(x),\quad \theta(x,-s)|_{s>0}=p_0(x,s),\quad
\xi(x,0)=\xi_0(x),\quad \xi(x,-s)|_{s>0}=q_0(x,s),
\end{array} \right.
\end{equation}
where $\varphi_0,\Phi_0,\psi_0,\Psi_0,\theta_0,p_0,\xi_0,q_0$ are prescribed data.
We consider either the mixed Dirichlet-Neumann boundary conditions
\begin{equation}
\label{BCM-beg}
\varphi(0,t) = \varphi(L,t) = \psi_x(0,t) = \psi_x(L,t) = \theta_x(0,t) = \theta_x(L,t) = \xi(0,t) = \xi(L,t)=0,
\end{equation}
or else the full Dirichlet boundary conditions
\begin{equation}
\label{BCD-beg}
\varphi(0,t) = \varphi(L,t) = \psi(0,t) = \psi(L,t) = \theta(0,t) = \theta(L,t) = \xi(0,t) = \xi(L,t)=0.
\end{equation}
As detailed in the sequel, the treatment of the boundary conditions \eqref{BCD-beg} is harder than
\eqref{BCM-beg} and constitutes one of the main challenges of the article.

The aim of the present paper is to study the asymptotic properties of the solution semigroup $S(t)$
associated with \eqref{tim4-beg}-\eqref{BCD-beg} in the history space framework of Dafermos \cite{DAF}.
Before describing our main results,
we briefly summarize some recent achievements on related models where different
thermal laws have been employed.

\subsection{The Fourier law}

When the Gurtin-Pipkin laws \eqref{GP} are replaced by the classical Fourier ones
\begin{equation}
\label{sysfou}
\left\{\begin{array}{lcl}
\rho_3 \theta_{t} -\varpi_1 \theta_{xx}
+ \gamma (\varphi_{x} + \psi)_t = 0,\\\noalign{\vskip0.7mm}
\displaystyle	\rho_4 \xi_{t} -
\varpi_2\xi_{xx}+ \sigma \psi_{xt} = 0,
\end{array} \right.
\end{equation}
we obtain the so-called Timoshenko-Fourier system with full thermal coupling, whose
stability properties have been recently studied in
\cite{alvesetal-JEE}. In that paper, for a wide range of boundary conditions including
\eqref{BCM-beg} and \eqref{BCD-beg}, it is shown that the associated solution semigroup is exponentially
stable independently of the values of the structural parameters of the model. The main reason why no constraints
on the coefficients are needed to get exponential stability lies in the fact that
the system is {\it fully} damped, i.e.\ all the variables in play are effectively damped via the thermal dissipation.
Instead, when the system is only partially damped (i.e.\
the effects of either $\theta$ or else $\xi$ are neglected)
exponential stability occurs only within
the equal wave speed assumption $\rho_1 b  = \rho_2 k$ (see \cite{ASR,CARD,RR}).

\subsection{The Cattaneo law}

As is well-known, the Fourier heat conduction law has a parabolic character
and predicts that thermal signals propagate with an infinite
speed (see e.g.\ \cite{CJ}). In order to correct this unphysical phenomenon, several alternative theories
have been proposed. One of them is due to Cattaneo \cite{CATT} and consists in introducing
a (small) thermal relaxation parameter allowing to make
the resulting equation hyperbolic. Considering the Cattaneo
law in our model means replacing \eqref{GP} with
\begin{equation}
\label{catta}
\left\{\begin{array}{lcl}
\rho_3 \theta_t  + q_x+ \gamma (\varphi_{x} + \psi)_t =0,\\\noalign{\vskip0.7mm}
\tau q_t + q + \varpi_1 \theta_x=0,\\\noalign{\vskip0.7mm}
\rho_4 \xi_t + p_x + \sigma \psi_{xt}=0,\\\noalign{\vskip0.7mm}
\varsigma p_t + p + \varpi_2 \xi_x=0,
\end{array} \right.
\end{equation}
where $q=q(x,t)$ and $p=p(x,t)$ are the so-called heat-flux variables and
$\tau,\varsigma>0$ represent the aforementioned thermal relaxation parameters. Note that the system
above reduces to \eqref{sysfou} in the limit situation when $\tau=\varsigma =0$.
The stability properties of the resulting Timoshenko-Cattaneo model with full
thermal coupling have been recently analyzed in
\cite{DJ}, where it is proved that
the associated solution semigroup is exponentially
stable independently of the values of the structural parameters.
As in the Fourier case, this happens because
the system is fully damped, and indeed when the effects of either $\theta$ or else $\xi$
are neglected exponential stability holds only within appropriate conditions that somehow
generalize the equal wave speed assumption (see \cite{SJR}).

\subsection{Our results}
As our main result, we show that the semigroup $S(t)$ associated with \eqref{tim4-beg}-\eqref{BCD-beg}
is exponentially stable
independently of the values of the structural parameters of the model. Since the Cattaneo
law can be seen as the particular instance of the Gurtin-Pipkin one corresponding to the choices
$$g(s)=\frac{1}{\tau}\e^{-\frac{s}{\tau}}\qquad \text{and} \qquad
h(s)= \frac{1}{\varsigma}\e^{-\frac{s}{\varsigma}},$$
the exponential stability of the Timoshenko-Cattaneo system follows as a special case
(see the final Section \ref{finsec} for more details). Even more so,
the Timoshenko-Fourier system can be recovered from the
Timoshenko-Gurtin-Pipkin one through a proper singular limit procedure,
where the kernels $g$ and $h$ collapse into the Dirac mass at zero (see again
Section \ref{finsec} for more details).

As in the Fourier and the Cattaneo cases, the fact that system \eqref{tim4-beg} is fully damped
allows us to achieve the exponential stability without any restriction on the structural parameters
of the model,
contrarily to what happens in the partially damped situation where appropriate
stability conditions are needed (see \cite{delloro-pata}).
Still, the main challenge encountered in our analysis is connected to the treatment of the full
Dirichlet boundary conditions \eqref{BCD-beg} which produce some
``pointwise" boundary terms in the estimates.
Such terms have been handled in \cite{alvesetal-JEE} by means of a general observability inequality
recently established in \cite{AL,MUAV}, combined with some localized estimates obtained
by means of appropriate cut-off functions. This method heavily relies on the
regularization of the temperature variables provided by the parabolicity of the heat equation,
and thus cannot be applied to
\eqref{tim4-beg}-\eqref{BCD-beg} due to the  hyperbolic character of the Gurtin-Pipkin thermal law.
Hence, specific observability-type inequalities are needed to treat our problem.

\subsection{Plan of the paper}
In the forthcoming Section \ref{S2} we introduce the functional setting and the notation,
while in the subsequent Section \ref{semsec} we deal with the existence of the solution semigroup
$S(t)$. In Section \ref{MAINSEC} we state and prove the main result of the article.
The final Section \ref{finsec} is devoted to some concluding remarks.


\section{Functional Setting and Notation}
\label{S2}

We denote by $\R^+=(0,\infty)$
the positive half-line and by
$\i\R$ the imaginary axis in the complex plane.
The symbols $L^2, H^1,H_0^1$ and $H^2$ denote the standard (complex)
Lebesgue and Sobolev spaces on $(0,L)$, while
$\langle\cdot,\cdot\rangle$ denotes the
standard inner product on $L^2$, with associated norm $\|\cdot\|$.
We also introduce the spaces
$$
L^2_*=\big\{ f\in L^2 : \int_0^L f(x) \d x = 0\big\}\qquad \text{and}\qquad
H^1_* = H^1\cap L^2_*,
$$
the latter equipped with the gradient norm.
Concerning the convolution kernels $g$ and $h$, we suppose that for $s\geq0$
$$
g(s) = \int_s^\infty \mu(r) \d r,\qquad h(s) = \int_s^\infty \nu(r) \d r,
$$
where the so-called memory kernels $\mu,\nu:\R^+\to[0,\infty)$ are
non-increasing absolutely continuous functions, possibly unbounded near zero. Note that
$\mu$ and $\nu$ are integrable with total mass $g(0)$ and $h(0)$, respectively,
and that are differentiable almost everywhere with non-positive derivative.
They are also required to satisfy the conditions
\begin{align}
\label{assdafmu}
\mu'(s) + \delta_1\,\mu(s) &\leq 0,\\
\label{assdafnu}
\nu'(s) + \delta_2\,\nu(s) &\leq 0,
\end{align}
for some $\delta_1,\delta_2>0$ and almost every $s>0$.
Next, we introduce the
so-called memory spaces
\begin{align*}
&\M = \begin{cases} L^2_{\mu}(\R^+; H^1_*)\qquad \text{(b.c.\ \eqref{BCM-beg})}\\
L^2_{\mu}(\R^+; H^1_0) \qquad \text{(b.c.\ \eqref{BCD-beg})}\end{cases}\qquad\text{and}
\qquad \N = L^2_{\nu}(\R^+; H^1_0)
\end{align*}
endowed with the inner products
$$
\langle\eta_1,\eta_2\rangle_{\M}=\int_0^\infty \mu(s) \langle\eta_{1x}(s),\eta_{2x}(s)\rangle \d s,\qquad
\langle\zeta_1,\zeta_2\rangle_{\N}=\int_0^\infty \nu(s) \langle\zeta_{1x}(s),\zeta_{2x}(s)\rangle \d s.
$$
The induced norms will be denoted by $\|\cdot\|_{\M}$ and $\|\cdot\|_{\N}$, respectively.
Finally, we define the state space
$$
\H = \begin{cases}
H_0^1 \times L^2 \times  H_*^1 \times L^2_*
\times L^2_* \times \M \times L^2 \times \N\qquad \text{(b.c.\ \eqref{BCM-beg})}\\
 H_0^1 \times L^2 \times  H_0^1 \times L^2
\times L^2 \times \M \times L^2 \times \N \qquad \text{(b.c.\ \eqref{BCD-beg})}
\end{cases}
$$
equipped with the inner product
\begin{align*}
\langle U, \tilde U \rangle_{\cal H} &= k \langle \varphi_x+\psi, \tilde{\varphi}_x+\tilde{\psi} \rangle + \rho_1 \langle \Phi, \tilde{\Phi} \rangle + b\langle \psi_x , \tilde{\psi}_x \rangle + \rho_2\langle \Psi, \tilde{\Psi} \rangle \\
&\quad + \rho_3\langle \theta, \tilde{\theta} \rangle+ \varpi_1 \langle \eta,\tilde{\eta} \rangle_{\cal M} + \rho_4 \langle \xi, \tilde{\xi} \rangle + \varpi_2\langle \zeta, \tilde{\zeta} \rangle_{\cal N}
\end{align*}
for every $U = (\varphi,\Phi,\psi,\Psi,\theta,\eta,\xi,\zeta) \mbox{ and } \tilde U= (\tilde{\varphi},
\tilde{\Phi},\tilde{\psi},\tilde{\Psi},\tilde{\theta},\tilde{\eta},\tilde{\xi},\tilde{\zeta})$ belonging to $\H$.
The induced
norm, equivalent to the standard product norm, will be denoted by
$\| \cdot\|_\H$ and reads
$$
\| U\|_\H^2 =
k \|\varphi_x+\psi\|^2 + \rho_1 \| \Phi \|^2
+ b\|\psi_x\|^2 + \rho_2 \|\Psi\|^2
+ \rho_3\|\theta\|^2 +
\varpi_1 \|\eta\|^2_{\cal M} + \rho_4 \| \xi\|^2 +
\varpi_2\|\zeta\|_{\cal N}.
$$

\noindent
{\bf A word of warning.}
{\it Along the paper, we will make use of the Young, H\"older
and Poincar\'e inequalities without explicit mention.
We will also tacitly employ the equivalence between the norm $\| \cdot\|_\H$
and the standard product norm on the space $\H$.}


\section{The Semigroup}
\label{semsec}
We consider the infinitesimal generator $T$ of the right-translation semigroup on $\M$,
that is, the linear operator
$$T \eta=-\eta' \qquad\, \text{with domain}\qquad\, \D(T)=\big\{\eta\in{\M}:\eta'\in\M\,\,\,\text{and}\,\,\,
\lim_{s\to 0}\|\eta_x(s)\|=0\big\},$$
where $\eta'$ is the weak derivative with respect to $s\in\R^+$. We will also consider
the infinitesimal generator of the right-translation semigroup on $\N$, denoted again by
$T$ and defined in exactly the same way.
Calling for every $\eta,\zeta \in \D(T)$
$$
\Gamma[\eta] = \int_0^\infty -\mu'(s) \|\eta_x(s)\|^2 \d s \qquad \text{and} \qquad
\Gamma[\zeta] = \int_0^\infty -\nu'(s) \|\zeta_x(s)\|^2 \d s,
$$
we have the equalities (see e.g.\ \cite{Terreni})
\begin{align}
\label{Teta}
\Re \l T\eta,\eta\r_\M
= -\frac12 \Gamma[\eta] &\leq0,\\\noalign{\vskip1mm}
\label{Tzeta}
\Re \l T\zeta,\zeta\r_\N = -\frac12 \Gamma[\zeta] &\leq0.
\end{align}
Next, in the same spirit of \cite{DAF},
we define for $s>0$ the auxiliary variables
$$
\eta^t(x,s)=\int_{0}^s \theta(x,t-r)\d r\qquad \text{and} \qquad
\zeta^t(x,s)=\int_{0}^s \xi(x,t-r)\d r.$$
Note that within the mixed Dirichlet-Neumann boundary conditions \eqref{BCM-beg} the variables
$\eta$ and $\zeta$ satisfy the boundary conditions
$$\eta_x^t(0,s) = \eta_x^t(L,s) = \zeta^t(0,s) = \zeta^t(L,s)=0,$$
while within the full Dirichlet boundary conditions \eqref{BCD-beg}
the variables
$\eta$ and $\zeta$ satisfy the boundary conditions
$$
\eta^t(0,s) = \eta^t(L,s) = \zeta^t(0,s) = \zeta^t(L,s)=0.
$$
At this point, we rewrite \eqref{tim4-beg} as
\begin{equation}
\label{probrew}
\left\{\begin{array}{lcl}
\rho_{1} \varphi_{tt} -  k(\varphi_{x} + \psi)_{x} + \gamma \theta_x = 0,
\\ \noalign{\vskip3.5mm}
\rho_{2} \psi_{tt} - b\psi_{xx} + k(\varphi_{x} + \psi) -\gamma \theta + \sigma \xi_x = 0,
\\ \noalign{\vskip2.5mm}
\displaystyle	\rho_3 \theta_{t} -\varpi_1\int_{0}^{\infty} \mu(s) \eta_{xx}(s) \mathrm{d} s
+ \gamma (\varphi_{x} + \psi)_t = 0, \\ \noalign{\vskip1.7mm}
\eta_t - T\eta - \theta = 0, \\ \noalign{\vskip1.7mm}
\displaystyle	\rho_4 \xi_{t} -\varpi_2\int_{0}^{\infty}
\nu(s) \zeta_{xx}(s) \mathrm{d} s + \sigma \psi_{xt} = 0,\\ \noalign{\vskip1.7mm}
\zeta_t - T\zeta - \xi = 0.
\end{array} \right.
\end{equation}
The initial conditions \eqref{init-data} translate into
\begin{equation}\label{init-data-full}
\left\{\begin{array}{lcl}
\varphi(x,0)= \varphi_0(x),\quad \varphi_t(x,0)= \Phi_0(x),\quad
\psi(x,0)= \psi_0(x),\quad \psi_t(x,0)= \Psi_0(x), \smallskip\\
\theta(x,0)=\theta_0(x),\quad \eta^0(x,s)=\eta_0(x,s),\quad
\xi(x,0)=\xi_0(x),\quad \zeta^0(x,s)=\zeta_0(x,s),
\end{array} \right.
\end{equation}
where $\eta_0(x,s) = \int_0^s p_0(x,r) \mathrm{d} r$ and $\zeta_0(x,s) = \int_0^s q_0(x,r) \mathrm{d} r$.
Introducing now the state vector
$$U(t)= (\varphi(t),\Phi(t),\psi(t),\Psi(t),\theta(t),\eta^t,\xi(t),\zeta^t)\in\H,$$
we view problem \eqref{probrew}-\eqref{init-data-full} as the
abstract first-order ODE
\begin{equation}\label{refabs}
	\left \{
	\begin{array}{l}
		U_t= \mathcal{A} U, \quad t>0,\smallskip \\
		U(0)=U_0,
	\end{array}
	\right.
\end{equation}
where
$U_0 = (\varphi_{0},\Phi_{0},\psi_{0},\Psi_{0},\theta_0,\eta_0,\xi_0,\zeta_0)\in\H$
and the operator $\mathcal{A}:\D(\mathcal{A}) \subset \H \to \H$ reads
$$
\mathcal{A}\left(\begin{matrix}
\varphi\\
\Phi\\
\psi\\
\Psi\\
\theta\\
\eta\\
\xi\\
\zeta
\end{matrix}
\right)
=\left(
\begin{matrix}
\Phi\\
\frac{k}{\rho_1} (\varphi_x + \psi)_x -\frac{\gamma}{\rho_1}\theta_x\\
\Psi\\
\frac{b}{\rho_2}\psi_{xx} - \frac{k}{\rho_2}(\varphi_x+\psi)+\frac{\gamma}{\rho_2}\theta -\frac{\sigma}{\rho_2}\xi_x\\
\noalign{\vskip1.5mm}
\frac{\varpi_1}{\rho_3}\int_0^\infty \mu(s) \eta_{xx}(s)\d s - \frac{\gamma}{\rho_3}(\Phi_x+\Psi)\\
\noalign{\vskip.5mm}
T\eta + \theta\\
\noalign{\vskip.5mm}
\frac{\varpi_2}{\rho_4}\int_0^\infty \nu(s)\zeta_{xx}(s)\d s - \frac{\sigma}{\rho_4}\Psi_x\\
\noalign{\vskip.5mm}
T \zeta + \xi
\end{matrix}
\right).
$$
The domain of $\mathcal{A}$ is defined as
$$\D(\mathcal{A}) = \begin{cases}
\{U\in\mathcal{W} \;|\; \psi_x, \int_0^\infty \mu(s)\eta_x(s)\d s \in H_0^1;
\Psi ,\theta \in H_*^1 \}
\qquad &\text{(b.c.\ \eqref{BCM-beg})}
\\\noalign{\vskip0.3mm}
\{U\in\mathcal{W} \;|\;\psi\in  H^2; \Psi ,\theta \in H_0^1;
\int_0^\infty \mu(s)\eta(s)\d s \in H^2 \} \qquad &\text{(b.c.\ \eqref{BCD-beg})}
\end{cases}
$$
where
$$
\mathcal{W} =
\big\{U\in\H \;|\;
\varphi \in H^2;
\Phi,\xi \in  H_0^1;
\int_0^\infty \nu(s)\zeta(s)\d s \in H^2;
\eta,\zeta \in \D(T)\big\}.
$$
With the aid of \eqref{Teta}-\eqref{Tzeta},
after a standard computation we get the equality
\begin{equation}
\label{dissip}
\Re \l \mathcal{A} U, U\r_\H =
-\frac{\varpi_1}{2} \Gamma[\eta] -
\frac{\varpi_2}{2} \Gamma[\zeta]\leq0, \quad\, \forall U \in \D(\mathcal{A}),
\end{equation}
so that $\mathcal{A}$ is dissipative.
By means of standard techniques (see e.g.\ \cite{viscofraz,Pat}), one can also prove that
$I-\mathcal{A}$ is surjective. Thus $\mathcal{A}$ is densely defined and, due
to the Lumer-Phillips theorem, it is
the infinitesimal generator of a contraction $C_0$-semigroup $S(t):\mathcal{H}\to\mathcal{H}$ (see e.g.\ \cite{pazy}).
In particular:
\begin{itemize}
\item if $U_0\in{{\cal H}},$ then problem  $ (\ref{refabs}) $ has
a unique mild solution $U\in C^0([0,\infty),{{\cal H}})$ given by
$$U(t)=S(t)U_0, \quad t\geq0;$$
\item  if $U_0\in \D({\cal A}),$  then problem  $ (\ref{refabs}) $ has a unique classical solution
$$U\in C^0([0,\infty),\D({\cal A}))\cap  C^{1}([0,\infty),{\cal{H}});$$
\item  if $U_0\in D({{\cal A}}^n)$ for some $n\geq2$, then the solution is more regular, that is
$$ U\in \bigcap_{\ell=0}^{n}C^{n-\ell}([0,\infty),D({\cal A}^{\ell})).$$
\end{itemize}


\section{Exponential Stability}
\label{MAINSEC}

The main result of the paper reads as follows:

\begin{Theorem}
\label{EXP-STAB-TEO}
The contraction $C_0$-semigroup $S(t):\H\to\H$ generated by $\mathcal{A}$ is exponentially stable, namely,
there exist two structural constants\hspace{0.2mm} $\omega>0$ and $K=K(\omega)\geq1$ such that
$$
\|S(t)\|_{\mathcal{L}(\mathcal{H})} \leq K \e^{-\omega t}, \quad \,  \forall t\geq0.
$$
\end{Theorem}

The remaining of the section is devoted to the proof of Theorem \ref{EXP-STAB-TEO}.

\subsection{Resolvent analysis}
For every $\lambda \in \R$ and
$\widehat U=(\hat \varphi,\hat \Phi,\hat \psi,\hat \Psi,\hat\theta,\hat\eta,\hat\xi,\hat\zeta)\in \H$,
we consider the resolvent equation
$$
\i\lambda U - \mathcal{A} U = \widehat U
$$
in the unknown $U = (\varphi,\Phi,\psi,\Psi,\theta,\eta,\xi,\zeta) \in \D(\mathcal{A})$.
Multiplying by $ 2U$ in $\H$, taking the real part and exploiting \eqref{dissip}, we get the identity
$$
\varpi_1 \Gamma[\eta] +
\varpi_2 \Gamma[\zeta] = 2 \Re \l \i\lambda U - \mathcal{A} U, U\r_\H =
2\Re \l \widehat U , U\r_\H.
$$
Recalling that $\Gamma[\eta]\geq0$ and $ \Gamma[\zeta]\geq0$, we readily find
\begin{equation}
\label{DISS}
\varpi_1 \Gamma[\eta] +\varpi_2 \Gamma[\zeta]  \leq 2 \|U\|_\H \| \widehat U\|_\H.
\end{equation}
Such an estimate yields the following bound on the memory variables $\eta$ and $\zeta$.

\begin{Lemma}
\label{ETA+ZETA}
For every $\lambda\in \mathbb{R}$, the inequality
$$
\varpi_1 \|\eta\|_{\M}^2 + \varpi_2 \|\zeta \|_{\N}^2 \leq c\|U\|_\H \|\widehat U\|_\H
$$
holds for some structural constant $c>0$ independent of $\lambda$.
\end{Lemma}

\begin{proof}
Follows immediately from \eqref{DISS} and \eqref{assdafmu}-\eqref{assdafnu}.
\end{proof}

At this point, we write the resolvent equation componentwise:
\begin{align}
\label{R1}
&\i\lambda \varphi - \Phi =\hat \varphi,\\\noalign{\vskip0.7mm}
\label{R2}
&\i\lambda \rho_1 \Phi  -k(\varphi_{x} +\psi)_x  +\gamma \theta_x =\rho_1 \hat \Phi,\\\noalign{\vskip0.7mm}
\label{R3}
&\i\lambda \psi - \Psi =\hat \psi,\\\noalign{\vskip0.7mm}
\label{R4}
&\i\lambda\rho_2 \Psi -b\psi_{xx} +k(\varphi_{x} +\psi) -\gamma\theta + \sigma \xi_x=\rho_2 \hat \Psi,\\
\label{R5}
&\i\lambda \rho_3 \theta - \varpi_1\int_0^\infty \mu(s)\eta_{xx}(s)\d s
+ \gamma (\Phi_{x} +\Psi) = \rho_3 \hat\theta,\\
\label{R6}
&\i\lambda \eta - T\eta - \theta = \hat\eta,\\
\label{R7}
&\i\lambda \rho_4 \xi - \varpi_2\int_0^\infty \nu(s)\zeta_{xx}(s)\d s + \sigma \Psi_{x}= \rho_4 \hat\xi,\\
\label{R8}
&\i\lambda \zeta - T\zeta - \xi= \hat\zeta.
\end{align}
In the next two results, we establish suitable controls on the temperature variables $\theta$ and $\xi$.

\begin{Lemma}
\label{THETA+XI}
For every $\lambda\in \mathbb{R}$ and every $\eps \in (0,1)$, the inequality
$$
\rho_3\|\theta\|^2 + \rho_4 \|\xi\|^2 \leq \eps\|U\|_\H^2 + \frac{c}{\eps} \|U\|_\H\|\widehat U\|_\H
$$
holds for some structural constant $c>0$ independent of $\lambda$ and $\eps$.
\end{Lemma}

In the proof of Lemma \ref{THETA+XI}, as well as in the proofs of the subsequent
Lemmas \ref{Theta-aux}-\ref{corol-varphix-psi-final}, we always denote by $c>0$ a generic structural constant
independent of $\lambda$, whose value
might change from line to line or even within the same line.

\begin{proof}
In order to deal with
the possible singularity of $\mu$ at zero, we fix $s_0>0$ such that $\mu(s_0)>0$ and
we define the kernel $m(s)= \mu(s_0)\chi_{(0,s_0]}(s) + \mu(s)\chi_{(s_0,\infty)}(s) $.
Then, we consider the space
\begin{align*}
&\mathcal{U}_0= \begin{cases} L^2_{m}(\R^+; L^2_*)\qquad \text{(b.c.\ \eqref{BCM-beg})}\\
L^2_{m}(\R^+; L^2) \qquad \text{(b.c.\ \eqref{BCD-beg})}\end{cases}
\end{align*}
equipped with the inner product
$$
\langle\eta_1,\eta_2\rangle_{\mathcal{U}_0}=\int_0^\infty m(s) \langle\eta_{1}(s),\eta_{2}(s)\rangle \d s.
$$
Since $m(s)\leq \mu(s)$, the memory space $\M$ is continuously embedded into $\mathcal{U}_0$.
Therefore, we can
multiply \eqref{R6} by $\rho_3\theta$ in $\mathcal{U}_0$, finding the identity
\begin{equation}
\label{temp1}
\rho_3\Big(\int_0^\infty m(s) \d s\Big)\|\theta\|^2
= \underbrace{\i\lambda \rho_3 \l\eta, \theta\r_{\mathcal{U}_0}}_{:=I_1}
\underbrace{- \rho_3\l T\eta, \theta\r_{\mathcal{U}_0}}_{:=I_2}-\rho_3 \l \hat \eta, \theta\r_{\mathcal{U}_0}.
\end{equation}
Exploiting \eqref{R5}, it is not difficult to see that
\begin{align*}
|I_1|
&\leq c\Big(\int_0^\infty \mu (s)\|\eta_x(s)\| \d s\Big)^2+
c\Big(\int_0^\infty \mu (s)\|\eta_x(s)\| \d s\Big) \|\Phi\|  +c|\l \eta,  \Psi\r_{\mathcal{U}_0}|
+c|\l \eta,\hat\theta \r_{\mathcal{U}_0}|\\\noalign{\vskip0.7mm}
&\leq c \|\eta\|_{\M}^2 +  c \|\eta\|_{\M} \|\Phi\| + c \|\eta\|_{\M} \|\Psi\|
+ c \|\eta\|_{\M} \|\hat \theta\|\\ \noalign{\vskip1.8mm}
& \leq   c \|\eta\|_{\M}^2 + c \|\eta\|_{\M} \|U\|_\H  + c\|U\|_\H \|\widehat U\|_\H.
\end{align*}
Integrating by parts in $s$ (the boundary terms vanish, see e.g.\ \cite{Terreni}),
we also infer that
$$
|I_2| = \big| \rho_3 \int_{s_0}^\infty -\mu'(s)\l \eta(s), \theta \r \d s \big|
\leq c\|\theta\|\Big(\int_{s_0}^\infty -\mu'(s) \|\eta_x(s)\| \d s \Big) \leq c \|U\|_\H \sqrt{\Gamma[\eta]}.
$$
Plugging the estimates above into \eqref{temp1} and
noting that $|-\rho_3 \l \hat \eta, \theta\r_{\mathcal{U}_0}|\leq c\|U\|_\H \|\widehat U\|_\H$,
we obtain
\begin{align*}
\rho_3\|\theta\|^2 &\leq c \|\eta\|_{\M}^2 + c \|\eta\|_{\M} \|U\|_\H  + c\|U\|_\H \|\widehat U\|_\H
+ c \|U\|_\H \sqrt{\Gamma[\eta]}\\
&\leq  c\|U\|_\H \|\widehat U\|_\H + c \|U\|_\H \sqrt{\|U\|_\H \|\widehat U\|_\H }
\end{align*}
where the second inequality follows from \eqref{DISS} and Lemma \ref{ETA+ZETA}. Thus,
for every $\eps\in (0,1)$, we end up with
$$
\rho_3\|\theta\|^2 \leq \frac{\eps}{2} \|U\|_\H^2 + \frac{c}{\eps} \|U\|_\H\|\widehat U\|_\H
$$
where $c>0$ is independent of $\lambda$ and $\eps$.

Next, in order to deal with
the possible singularity of $\nu$ at zero, we fix $s_1>0$ such that $\nu(s_1)>0$ and
we introduce the kernel $n(s)= \nu(s_1)\chi_{(0,s_1]}(s) + \nu(s)\chi_{(s_1,\infty)}(s)$.
Then, we consider the space
$$
\mathcal{V}_0= L_n^2(\mathbb{R}^+; L^2),
$$
equipped with the inner product
$$
\langle\zeta_1,\zeta_2\rangle_{\mathcal{V}_0}=\int_0^\infty n(s) \langle\zeta_{1}(s),\zeta_{2}(s)\rangle \d s.
$$
Being $n(s)\leq\nu(s)$, the memory space $\N$ is continuously embedded into $\mathcal{V}_0$. As a consequence,
multiplying \eqref{R8} by $\rho_4 \xi$ in $\mathcal{V}_0$, we get
\begin{equation}
\label{temp1-3}
\rho_4\Big(\int_0^\infty n(s) \d s\Big)\|\xi\|^2= \underbrace{\i\lambda \rho_4 \l\zeta, \xi\r_{\mathcal{V}_0}}_{:=I'_1}
\underbrace{- \rho_4\l T\zeta, \xi\r_{\mathcal{V}_0}}_{:=I'_2}-\rho_4 \l \hat \zeta, \xi\r_{\mathcal{V}_0}.
\end{equation}
An exploitation of \eqref{R7} yields
$$
|I'_1| \leq   c \|\zeta\|_{\N}^2 + c \|\zeta\|_{\N} \|U\|_\H  + c\|U\|_\H \|\widehat U\|_\H,
$$
while integrating by parts in $s$
we find
$$|I'_2|\leq c\|U\|_{\cal{H}}\sqrt{\Gamma[\zeta]}$$
(cf.\ the corresponding estimates for $I_1$ and $I_2$ above).
Plugging these inequalities
into \eqref{temp1-3} and owing to \eqref{DISS} and Lemma \ref{ETA+ZETA}, we finally get
\begin{align*}
\rho_4\|\xi\|^2 &\leq c \|\zeta\|_{\N}^2 + c \|\zeta\|_{\N} \|U\|_\H
+ c\|U\|_\H \|\widehat U\|_\H + c \|U\|_\H \sqrt{\Gamma[\zeta]}\\
&\leq  c\|U\|_\H \|\widehat U\|_\H + c \|U\|_\H \sqrt{\|U\|_\H \|\widehat U\|_\H}\\\noalign{\vskip1mm}
&\leq \frac{\eps}{2} \|U\|_\H^2 + \frac{c}{\eps} \|U\|_\H\|\widehat U\|_\H
\end{align*}
for every $\eps\in (0,1)$ and some $c>0$ independent of $\lambda$ and $\eps$.
The proof is finished.
\end{proof}

\begin{Lemma}\label{Theta-aux}
For every $\lambda\in \mathbb{R}$, the inequality
$$
\|\theta_x\| + \|\xi_x\|\leq c \big[1+|\lambda|\big] \sqrt{\|U\|_\H \|\widehat U\|_\H} + c\| \widehat U\|_\H
$$
holds for some structural constant $c>0$ independent of $\lambda$.
\end{Lemma}

\begin{proof}
As in the proof of Lemma \ref{THETA+XI},
we consider the kernel $m(s)= \mu(s_0)\chi_{(0,s_0]}(s) + \mu(s)\chi_{(s_0,\infty)}(s)$
where $s_0>0$ is such that $\mu(s_0)>0$.
Then, we introduce the space
\begin{align*}
&\mathcal{U}_1= \begin{cases} L^2_{m}(\R^+; H_*^1)\qquad \text{(b.c.\ \eqref{BCM-beg})}\\
L^2_{m}(\R^+; H_0^1) \qquad \text{(b.c.\ \eqref{BCD-beg})}\end{cases}
\end{align*}
equipped with the inner product
$$
\langle\eta_1,\eta_2\rangle_{\mathcal{U}_1}=\int_0^\infty m(s) \langle\eta_{1x}(s),\eta_{2x}(s)\rangle \d s.
$$
Again, since
$m(s)\leq \mu(s)$, the memory space $\M$ is continuously embedded into $\mathcal{U}_1$. Thus,
multiplying \eqref{R6} by $\theta$ in $\mathcal{U}_1$, we infer that
$$
\Big(\int_0^\infty m(s) \d s\Big)\|\theta_x\|^2
= \i\lambda \langle \eta, \theta \rangle_{\mathcal{U}_1} -\langle T\eta, \theta \rangle_{\mathcal{U}_1}
- \langle \hat{\eta}, \theta\rangle_{\mathcal{U}_1}.
$$
It follows from Lemma \ref{ETA+ZETA} that
\begin{equation*}
|\i\lambda \langle \eta,\theta\rangle_{\mathcal{U}_1}|
\leq c|\lambda| \|\theta_x\|\|\eta\|_{\mathcal{M}}
\leq c|\lambda| \|\theta_x\|\sqrt{\|U\|_{\cal H} \|\widehat U\|_{\cal H}}.
\end{equation*}
Moreover, integrating by parts in $s$ and using \eqref{DISS}, we can write
(cf.\ the estimate for $I_2$ in the proof of Lemma \ref{THETA+XI})
$$
|-\langle T\eta, \theta \rangle_{\mathcal{U}_1}|=\big| \int_{s_0}^\infty -\mu'(s)\l \eta_x(s), \theta_x \r \d s \big|
\leq c \|\theta_x\| \sqrt{\Gamma[\eta]} \leq c \|\theta_x\| \sqrt{\|U\|_{\cal H} \|\widehat U\|_{\cal H}}.
$$
Finally, it is easy to see that
$|-\langle \hat{\eta}, \theta\rangle_{\mathcal{U}_1}|\leq c\|\theta_x\|\|\widehat U\|_{\cal H},$
and the required bound for $\|\theta_x\|$ follows.

In order to estimate $\|\xi_x\|$, we proceed in an analogous way.
As in the proof of Lemma \ref{THETA+XI},
we consider the kernel $n(s)= \nu(s_1)\chi_{(0,s_1]}(s) + \nu(s)\chi_{(s_1,\infty)}(s)$
where $s_1>0$ is such that $\nu(s_1)>0$,
and we introduce the space
$$
\mathcal{V}_1= L^2_{n}(\R^+; H_0^1)
$$
equipped with the inner product
$$
\langle\zeta_1,\zeta_2\rangle_{\mathcal{V}_1}=\int_0^\infty n(s) \langle\zeta_{1x}(s),\zeta_{2x}(s)\rangle \d s.
$$
Since $\N$ is continuously embedded into $\mathcal{V}_1$,
multiplying \eqref{R8} by $\xi$ in $\mathcal{V}_1$ we find
$$
\Big(\int_0^\infty n(s) \d s\Big)\|\xi_x\|^2
= \i\lambda \langle \zeta, \xi \rangle_{\mathcal{V}_1}
-\langle T\zeta, \xi \rangle_{\mathcal{V}_1}
- \langle \hat{\zeta}, \xi\rangle_{\mathcal{V}_1}.
$$
Arguing exactly as above, the modulus of the right-hand side is less than or equal to
$$
c \big[1+|\lambda|\big] \|\xi_x\|\sqrt{\|U\|_\H \|\widehat U\|_\H} + c\|\xi_x\|\| \widehat U\|_\H,
$$
and the required bound for $\|\xi_x\|$ has been proved.
\end{proof}

The next step is to control the variables $\Phi$ and $\Psi$.

\begin{Lemma}\label{lemma-Phi}
For every $\lambda\in \R$ and every $\varepsilon\in (0,1)$, the inequality
\begin{equation*}
\rho_1\|\Phi\|^2 +\rho_2\|\Psi\|^2\leq c\varepsilon\|U\|_\H^2
+c\|U\|_\H\big[ \|\varphi_x+\psi\|+\|\psi_x\|\big] +\dfrac{c}{\varepsilon^3}\|U\|_\H \|\widehat{U}\|_\H
\end{equation*}
holds for some structural constant $c>0$ independent of $\lambda$ and $\varepsilon$.
\end{Lemma}

\begin{proof}
Multiplying \eqref{R2} by $\varphi$ in $L^2$ and exploiting \eqref{R1}, we readily infer that
\begin{equation*}\label{lema-Phi-eq2}
\rho_1 \|\Phi\|^2=	 -\rho_1 \langle \hat{\Phi}, \varphi \rangle- \rho_1\langle \Phi, \hat{\varphi}\rangle
 +k\langle \varphi_x+\psi, \varphi_x \rangle - \gamma\langle \theta, \varphi_x\rangle.
\end{equation*}
Invoking Lemma \ref{THETA+XI},
for every $\varepsilon\in (0,1)$ we have
\begin{align*}
\nonumber \rho_1\|\Phi\|^2&\leq c\|U\|_\H \|\widehat U\|_\H+ c\|U\|_\H\|\varphi_x+\psi\| +c\|U\|_\H\|\theta\| \\
\noalign{\vskip1mm}
\nonumber&\leq c\|U\|_\H \|\widehat U\|_\H +c\|U\|_\H\|\varphi_x+\psi\|
+c\|U\|_\H\left[\varepsilon \|U\|_\H +\dfrac{1}{\varepsilon}\sqrt{\|U\|_\H\|\widehat{U}\|_\H} \right]   \\
&\leq c\varepsilon\|U\|_\H^2 +c\|U\|_\H\|\varphi_x+\psi\|+\dfrac{c}{\varepsilon^3}\|U\|_\H \|\widehat{U}\|_\H,
\end{align*}
where $c>0$ is independent of $\lambda$ and $\eps$. Now, we multiply \eqref{R4} by $\psi$ in $L^2$
and we use \eqref{R3} to get
\begin{equation*}
\rho_2\|\Psi\|^2=-\rho_2 \langle\hat \Psi,\psi\rangle-\rho_2 \langle\Psi,\hat{\psi}\rangle  +b\|\psi_x\|^2
+k\langle\varphi_{x} +\psi,\psi\rangle -\gamma\langle\theta,\psi\rangle - \sigma \langle \xi,\psi_x\rangle.
\end{equation*}
Again, an exploitation of Lemma \ref{THETA+XI} yields
\begin{align*}
\rho_2\|\Psi\|^2&\leq c\|U\|_\H \|\widehat U\|_\H+ c\|U\|_\H\|\psi_x\|
+c\|U\|_\H\|\varphi_x+\psi\| +c\|U\|_\H\big[\|\theta\|+\|\xi\|\big]\\
\noalign{\vskip1mm}
&\leq c\|U\|_\H \|\widehat U\|_\H+c \|U\|_\H\big[\|\varphi_x+\psi\|
+\|\psi_x\| \big]+c\|U\|_\H\left[\varepsilon\|U\|_\H
+\dfrac{1}{\varepsilon}\sqrt{\|U\|_\H\|\widehat{U}\|_\H} \right]   \\
&\leq c \varepsilon\|U\|_\H^2 +c\|U\|_\H\big[\|\varphi_x+\psi\|
+\|\psi_x\|\big] +\dfrac{c}{\varepsilon^3}\|U\|_\H \|\widehat{U}\|_\H
\end{align*}
for every $\varepsilon\in (0,1)$, where as before $c>0$ is independent of $\lambda$ and $\eps$.
Taking the sum of the two estimates obtained so far, we reach the thesis.
\end{proof}

We now need to control the variables $\varphi_x+\psi$ and $\psi_x$. To this end,
we introduce the functions
$$
\alpha(x)=\int_{0}^{\infty}\mu(s)\eta_x(x,s)\d s\qquad \mbox{and}
\qquad \beta(x)=\int_{0}^{\infty}\nu(s)\zeta_x(x,s)\d s,
$$
and we set
\begin{align*}
\mathcal{P}(\varphi,\psi,\alpha) &=|\alpha(L)| |\varphi_x(L)+\psi(L)|+ |\alpha(0)| |\varphi_x(0)+\psi(0)|,\\
\mathcal{Q}(\psi,\beta) &= |\beta(L)| |\psi_x(L)|+ |\beta(0)| |\psi_x(0)|.
\end{align*}
For the b.c.\ \eqref{BCM-beg} one has
$\mathcal{P}(\varphi,\psi,\alpha)=
\mathcal{Q}(\psi,\beta)=0$, but this is not the case for the b.c.~\eqref{BCD-beg}.

\begin{Lemma}\label{lemma-varphi_x+psi}
For every $\lambda\neq0$ and every $\eps \in (0,1)$,
the inequality
$$
k\|\varphi_x+ \psi\|^2 + b \|\psi_x\|^2 \leq c\eps\|U\|_\H^2
+\frac{c}{\eps}\bigg[\frac{1}{|\lambda|^2}+1\bigg]\|U\|_\H\|\widehat U\|_\H
+\frac{c}{|\lambda|}\big[\mathcal{P}(\varphi,\psi,\alpha) + \mathcal{Q}(\psi,\beta)\big]
$$
holds for some structural constant $c>0$ independent of $\lambda$ and $\eps$ .
\end{Lemma}

\begin{proof}
Replacing \eqref{R1} and \eqref{R3} into \eqref{R5}, we find
$$
\i\lambda \rho_3\theta - \varpi_1 \alpha_x
+\i\lambda \gamma  (\varphi_x+\psi)=\rho_3 \hat{\theta}+\gamma (\hat{\varphi_x}+\hat{\psi}).
$$
Multiplying the identity above by $k(\varphi_x + \psi)$ in $L^2$, we obtain
\begin{align}
\label{mix}
\i \lambda \gamma  k\|\varphi_x + \psi\|^2
= \underbrace{\varpi_1 k\l \alpha_{x}, \varphi_x + \psi \r}_{:=J_1}
\underbrace{-\i\lambda \rho_3 k \l \theta , \varphi_x + \psi\r}_{:=J_2}
 + k \l \rho_3 \hat  \theta
+ \gamma (\hat \varphi_x  + \hat \psi), \varphi_x + \psi\r.
\end{align}
With the aid of \eqref{R2}, we rewrite $J_1$ as
$$
J_1 =
\varpi_1 \l \alpha , \rho_1 \hat \Phi - \gamma \theta_x -
\i \lambda \rho_1 \Phi  \r  + \varpi_1 k\, \alpha \overline{(\varphi_x+\psi)}\Big|_0^L.
$$
Invoking Lemmas \ref{ETA+ZETA} and \ref{Theta-aux}, it is readily seen that
\begin{align*}
|\l \alpha , \rho_1 \hat \Phi - \gamma \theta_x -
\i \lambda \rho_1 \Phi  \r | &\leq c\|\eta\|_\M \|\theta_x\|
+ c|\lambda| \|\eta\|_\M \|\Phi\| + c  \|U\|_\H\|\widehat U\|_\H \\
&\leq c \big[1+|\lambda|\big] \|U\|_\H \|\widehat U\|_\H + c|\lambda| \|\eta\|_\M \|U\|_\H,
\end{align*}
from where we get
$$
|J_1| \leq
c \big[1+|\lambda|\big] \|U\|_\H \|\widehat U\|_\H + c|\lambda| \|\eta\|_\M\|U\|_\H
+ c\hspace{0.4mm}\mathcal{P}(\varphi,\psi,\alpha).
$$
Since $|J_2| \leq c|\lambda|\|\theta\| \|\varphi_x+\psi\|$,
it follows from \eqref{mix} together with
Lemmas \ref{ETA+ZETA}-\ref{THETA+XI},  that
\begin{align*}
2k\|\varphi_x+\psi\|^2 &\leq c \bigg[\frac{1}{|\lambda|}+1\bigg]\|U\|_\H\|\widehat U\|_\H
+ c\|\eta\|_\M\|U\|_\H + c \|\theta\| \|\varphi_x+\psi\|
+\frac{c}{|\lambda|} \mathcal{P}(\varphi,\psi,\alpha) \\
&\leq \eps \|U\|_\H^2 +c \bigg[\frac{1}{|\lambda|}+1\bigg]\|U\|_\H\|\widehat U\|_\H
+ \frac{c}{\eps} \|\eta\|_\M^2 + c \|\theta\|^2 + k \|\varphi_x+\psi\|^2
+\frac{c}{|\lambda|} \mathcal{P}(\varphi,\psi,\alpha)
\\ &\leq c\eps \|U\|_\H^2 +\frac{c}{\eps} \bigg[\frac{1}{|\lambda|^2}+1\bigg]\|U\|_\H\|\widehat U\|_\H
+ k \|\varphi_x+\psi\|^2
+\frac{c}{|\lambda|} \mathcal{P}(\varphi,\psi,\alpha)
\end{align*}
for every $\lambda \neq0$ and $\eps\in(0,1)$, where
$c>0$ is independent of $\lambda$ and $\eps$.
In conclusion,
$$
k\|\varphi_x+ \psi\|^2  \leq c\eps\|U\|_\H^2
+\frac{c}{\eps}\bigg[\frac{1}{|\lambda|^2}+1\bigg]\|U\|_\H\|\widehat U\|_\H
+\frac{c}{|\lambda|} \mathcal{P}(\varphi,\psi,\alpha).
$$
In order to prove the analogous bound for $\|\psi_x\|$, we
substitute \eqref{R3} in \eqref{R7}, getting
$$
\i\lambda \rho_4\xi - \varpi_2\beta_x
+\i\lambda\sigma \psi_x=\rho_4\hat{\xi} +\sigma \hat{\psi}_x.
$$
Multiplying such an identity by $b\psi_x$ in $L^2$, we find
\begin{equation}\label{lemma-varphi_x+psi-eq4}
\i\lambda\sigma b\|\psi_x\|^2
=\underbrace{\varpi_2 b \langle \beta_{x}, \psi_x \rangle}_{:=J'_1}
\underbrace{-\i\lambda \rho_4 b \langle \xi, \psi_x \rangle}_{:=J'_2} +
\rho_4 b\langle \hat{\xi}, \psi_x\rangle + \sigma b\langle \hat{\psi}_x, \psi_x \rangle.
\end{equation}
Using \eqref{R4}, we rewrite
$$
J'_1 = \varpi_2 \l \beta, \rho_2 \hat \Psi
+ \gamma \theta-\sigma \xi_x- k(\varphi_x+\psi)-
\i \lambda \rho_2 \Psi  \r + \varpi_2 b \hspace{0.7mm} \beta \overline{\psi_x}\big|_0^L.
$$
An exploitation of Lemmas \ref{ETA+ZETA} and \ref{Theta-aux} now yields
(cf.\ the corresponding estimate for $J_1$ above)
$$
|J'_1| \leq c \big[1+|\lambda|\big] \|U\|_\H \|\widehat U\|_\H
+ c \big[1+|\lambda|\big] \|\zeta\|_\N\|U\|_\H+c\hspace{0.4mm} \mathcal{Q}(\psi,\beta).
$$
Since
$|J'_2| \leq c|\lambda|\|\xi\| \|\psi_x\|$,
making use of Lemmas \ref{ETA+ZETA}-\ref{THETA+XI} it
follows from \eqref{lemma-varphi_x+psi-eq4} that
\begin{align*}
2b\|\psi_x\|^2 &\leq c \bigg[\frac{1}{|\lambda|}+1\bigg]\|U\|_\H\|\widehat U\|_\H
+ c\bigg[\frac{1}{|\lambda|}+1\bigg]\|\zeta\|_\N\|U\|_\H
+ c \|\xi\| \|\psi_x\|+\frac{c}{|\lambda|} \mathcal{Q}(\psi,\beta)\\
&\leq \eps \|U\|_\H^2 +c \bigg[\frac{1}{|\lambda|}+1\bigg]\|U\|_\H\|\widehat U\|_\H
+ \frac{c}{\eps}\bigg[\frac{1}{|\lambda|^2}+1\bigg] \|\zeta\|_\N^2 + c \|\xi\|^2 + b \|\psi_x\|^2 +\frac{c}{|\lambda|} \mathcal{Q}(\psi,\beta)
\\ &\leq c\eps \|U\|_\H^2 +\frac{c}{\eps} \bigg[\frac{1}{|\lambda|^2}+1\bigg]\|U\|_\H\|\widehat U\|_\H + b \|\psi_x\|^2
+\frac{c}{|\lambda|} \mathcal{Q}(\psi,\beta),
\end{align*}
for every $\lambda \neq0$ and $\eps\in(0,1)$, where
$c>0$ is independent of $\lambda$ and $\eps$. Hence,
we end up with
$$
b\|\psi_x \|^2  \leq c\eps \|U\|_\H^2
+\frac{c}{\eps}\bigg[\frac{1}{|\lambda|^2}+1\bigg]\|U\|_\H\|\widehat U\|_\H
+\frac{c}{|\lambda|} \mathcal{Q}(\psi,\beta),
$$
leading to the desired conclusion.
\end{proof}

Our final task is to control the terms
$\mathcal{P}(\varphi,\psi,\alpha)$ and
$\mathcal{Q}(\psi,\beta)$ (within the b.c.\ \eqref{BCD-beg}).

\subsection{Observability analysis}

\begin{Lemma}[Elastic observability-type inequality]\label{lemma-term-sub-front}
Consider the full Dirichlet b.c.\ \eqref{BCD-beg}. For every $\lambda\in\R$, the
following inequalities hold for some structural constant
$c>0$ independent of $\lambda$.
\begin{itemize}
\item[(i)] Defining $\mathcal{I}(\varphi,\psi)=|\varphi_x(0)+\psi(0)|^2+|\varphi_x(L)+\psi(L)|^2$ we have
$$\mathcal{I}(\varphi,\psi)\leq c\|U\|_\H^2+
c\|U\|_\H \|\widehat{U}\|_\H + c\|U\|_\H \|\theta_x\|.$$
\item[(ii)] Defining $\mathcal{J}(\psi)=|\psi_x(0)|^2+|\psi_x(L)|^2$ we have
$$\mathcal{J}(\psi)\leq c\|U\|_\H^2 + c\|U\|_\H \|\widehat{U}\|_\H + c\|U\|_\H \|\xi_x\|.$$
\end{itemize}
\end{Lemma}

\begin{proof}
Setting $z(x)=(x-\sfrac{L}{2})$
we multiply \eqref{R2} by $z(\varphi_x+\psi)$ in $L^2$. Taking real part of the resulting equality, we obtain
$$
\underbrace{\Re\big[\i\lambda \rho_1 \langle \Phi, z (\varphi_x+\psi)\rangle
- k \langle (\varphi_x+\psi)_x, z (\varphi_x+\psi)\rangle\big]}_{:=P_1}
+\gamma \Re \langle \theta_x, z (\varphi_x+\psi) \rangle
= \rho_1\Re\langle \hat{\Phi}, z (\varphi_x+\psi)\rangle.
$$
Substituting ${\varphi}$ and ${\psi}$ given by \eqref{R1} and \eqref{R3} into $P_1$ and
after an elementary calculation, we find
$$
P_1 = \dfrac{\rho_1}{2}\|\Phi\|^2+\dfrac{k}{2}\|\varphi_x+\psi\|^2
- \rho_1\Re\langle \Phi, z \Psi\rangle
- \rho_1\Re\langle \Phi, z (\hat{\varphi}_x + \hat{\psi})\rangle
- \dfrac{kL}{4}\mathcal{I}(\varphi,\psi).
$$
Therefore, we get the identity
\begin{align*}
\mathcal{I}(\varphi,\psi) &= \dfrac{2\rho_1}{kL}\|\Phi\|^2+\dfrac{2}{L}\|\varphi_x+\psi\|^2
-\dfrac{4\rho_1}{kL}\Re\langle \Phi, z\Psi\rangle
 +\dfrac{4\gamma}{kL} \Re \langle \theta_x, z (\varphi_x+\psi)\rangle \\\noalign{\vskip0.8mm}
&\quad -\dfrac{4\rho_1}{kL} \Re\langle \hat{\Phi},z (\varphi_x+\psi)\rangle
- \dfrac{4\rho_1}{kL} \Re\langle \Phi,z (\hat{\varphi}_x + \hat{\psi})\rangle.
\end{align*}
Since the modulus of the right-hand side is less than or equal to
$$
c\|U\|_\H^2+
c\|U\|_\H \|\widehat{U}\|_\H + c\|U\|_\H\|\theta_x\|
$$
the proof of item (i) is finished.

In order to prove item (ii), we multiply \eqref{R4} by $z \psi_x$ in $L^2$.
Taking real part of the resulting identity, we arrive at
$$
\underbrace{\Re\big[\i \lambda \rho_2\langle\Psi, z\psi_x\rangle
-b\langle \psi_{xx}, z \psi_x\rangle\big]}_{:=P_2}
+k\Re\langle \varphi_x+\psi, z \psi_x\rangle
+  \Re\langle \sigma\xi_x - \gamma \theta, z \psi_x\rangle
= \rho_2\Re\langle \hat{\Psi}, z \psi_x\rangle.
$$
Inserting $\psi$ given by \eqref{R3} in $P_2$, after an elementary calculation we infer that
$$
P_2=\dfrac{\rho_2}{2}\|\Psi\|^2 +
\dfrac{b}{2}\|\psi_x\|^2-\rho_2\Re\langle \Psi,z\hat{\psi}_x\rangle
-\dfrac{b L}{4}\mathcal{J}(\psi).
$$
Thus, we have
\begin{align*}
\mathcal{J}(\psi) &=
\dfrac{2\rho_2}{b L}\|\Psi\|^2
+  \dfrac{2}{L}\|\psi_x\|^2 +\dfrac{4k}{bL}\Re\langle \varphi_x+\psi,z\psi_x\rangle
+ \dfrac{4\sigma}{bL}\Re\langle \xi_x , z\psi_x\rangle
\\\noalign{\vskip0.4mm}
&\quad - \dfrac{4\gamma}{bL}\Re\langle \theta , z\psi_x\rangle
-\dfrac{4\rho_2}{bL}\Re\langle \hat{\Psi}, z\psi_x\rangle
-\dfrac{4\rho_2}{bL}\Re\langle \Psi,z\hat{\psi}_x\rangle.
\end{align*}
Exploiting the equality above, we readily end up with the desired estimate
$$
\mathcal{J}(\psi)\leq c\|U\|_\H^2+
c\|U\|_\H \|\widehat{U}\|_\H + c\|U\|_\H\|\xi_x\|.
$$
The lemma has been proved.
\end{proof}

\begin{Lemma}[Viscoelastic observability-type inequality]
\label{lemma-term-front}
Consider the full Dirichlet b.c.\ \eqref{BCD-beg}.
For every $\lambda\in\R$ and every $\eps\in(0,1)$, the inequalities
\begin{align*}
\mathcal{P}(\varphi,\psi,\alpha)
&\leq c\eps\|U\|^2_\H + c\eps\|U\|_\H\|\theta_x\|
+ \dfrac{c}{\varepsilon}|\lambda|\|\eta\|_{\cal M}\|U\|_\H
+\dfrac{c}{\varepsilon}\|U\|_\H \|\widehat{U}\|_\H\\\noalign{\vskip0.3mm}
\mathcal{Q}(\psi,\beta)&\leq
c\eps\|U\|^2_\H + c\varepsilon \|U\|_\H \|\xi_x\|
+ \dfrac{c}{\varepsilon}|\lambda|\|\zeta\|_{\cal N}\|U\|_\H
 +\dfrac{c}{\varepsilon}\|U\|_\H\|\widehat{U}\|_\H	
\end{align*}
hold for some structural constant $c>0$ independent of $\lambda$ and $\eps$.
\end{Lemma}

\begin{proof}
Exploiting the Gagliardo-Nirenberg interpolation inequality (see e.g.\ \cite[p.\ 233]{brezis2010}),
we have
$$
\mathcal{P}(\varphi,\psi,\alpha)\leq c\|\alpha\|_{L^\infty(0,L)}
\sqrt{\mathcal{I}(\varphi,\psi)} \leq
\eps\mathcal{I}(\varphi,\psi) + \dfrac{c}{\varepsilon}\|\alpha\|_{L^\infty(0,L)}^2
 \leq \eps\mathcal{I}(\varphi,\psi) + \dfrac{c}{\varepsilon}\|\alpha\|^2
 +\dfrac{c}{\varepsilon}\|\alpha\| \|\alpha_x\|
$$
for every $\eps\in(0,1)$,
where $\mathcal{I}(\psi,\varphi)$ is given by Lemma \ref{lemma-term-sub-front}
and $c>0$ is independent of $\lambda$ and $\eps$.
On the other hand, combining equations \eqref{R1}, \eqref{R3} and \eqref{R5}, we can write
$$
\varpi_1 \alpha_x =
\i\lambda\rho_3\theta + \i\lambda \gamma(\varphi_x+\psi)
-\gamma(\hat{\varphi}_x+\hat{\psi})-\rho_3\hat{\theta},
$$
which yields the bound
$$
\|\alpha_x\|\leq c|\lambda|\big[\|\theta\| +\|\varphi_x+\psi\|\big] +c \|\widehat{U}\|_\H.
$$
Since $\|\alpha\|\leq c\|\eta\|_{\mathcal{M}}$, we finally obtain
\begin{align*}
\mathcal{P}(\varphi,\psi,\alpha) &\leq
\eps\mathcal{I}(\varphi,\psi) +\dfrac{c}{\varepsilon} \|\eta\|^2_{\cal M}
+\dfrac{c}{\varepsilon}|\lambda|\|\eta\|_{\cal M}\big[\|\theta\|
+\|\varphi_x+\psi\|\big]+\dfrac{c}{\varepsilon}\|\eta\|_{\cal M}\|\widehat{U}\|_\H\\\noalign{\vskip0.7mm}
&\leq c\eps\|U\|^2_\H + c\eps\|U\|_\H\|\theta_x\| + \dfrac{c}{\varepsilon}|\lambda|\|\eta\|_{\cal M}\|U\|_\H
+\dfrac{c}{\varepsilon}\|U\|_\H \|\widehat{U}\|_\H,
\end{align*}
where the second inequality follows from Lemmas \ref{ETA+ZETA} and \ref{lemma-term-sub-front}.

We are left to prove the analogous bound for $\mathcal{Q}(\psi,\beta)$. To this end,
exploiting again the Gagliardo-Nirenberg interpolation inequality and arguing exactly as above, we find
$$\mathcal{Q}(\psi,\beta)
\leq \varepsilon\mathcal{J}(\psi) + \dfrac{c}{\varepsilon}\|\beta\|^2
+\dfrac{c}{\varepsilon}\|\beta\| \|\beta_x\|
$$
for every $\eps\in(0,1)$, where $\mathcal{J}(\psi)$ is given by Lemma \ref{lemma-term-sub-front}
and $c>0$ is independent of $\lambda$ and $\eps$.
Additionally, combining equations \eqref{R3} and \eqref{R7}, we promptly have
\begin{equation*}\label{lemma-term-front-eq9}
\varpi_2 \beta_x= \i \lambda \rho_4\xi +\i\lambda\sigma \psi_x
- \sigma \hat{\psi}_x-\rho_4 \hat{\xi}.
\end{equation*}
Consequently, we can write
$$
\|\beta_x\| \leq c|\lambda|\big[\|\xi\| +\|\psi_x\|\big] +c \|\widehat{U}\|_\H.
$$
Due to the fact that $\|\beta\|\leq c\|\zeta\|_{\mathcal{N}}$, we end up with
\begin{align*}
\mathcal{Q}(\psi,\beta) &\leq
\eps\mathcal{J}(\psi) +\dfrac{c}{\varepsilon} \|\zeta\|^2_{\cal N}
+\dfrac{c}{\varepsilon}|\lambda|\|\zeta\|_{\cal N}\big[\|\xi\|
+\|\psi_x\|\big]+\dfrac{c}{\varepsilon}\|\zeta\|_{\cal N}\|\widehat{U}\|_\H\\\noalign{\vskip0.7mm}
&\leq c\eps\|U\|^2_\H + c\eps\|U\|_\H \|\xi_x\|
+ \dfrac{c}{\varepsilon}|\lambda|\|\zeta\|_{\cal N}\|U\|_\H
+\dfrac{c}{\varepsilon}\|U\|_\H \|\widehat{U}\|_\H,
\end{align*}
where the second inequality follows from Lemmas \ref{ETA+ZETA} and \ref{lemma-term-sub-front}.
The proof is over.
\end{proof}

We finally obtain the following estimate for the terms $\varphi_x+\psi$ and $\psi_x$.

\begin{Lemma}\label{corol-varphix-psi-final}
For every $\lambda\ne 0$ and every $\eps \in (0,1)$, the inequality
$$
k\|\varphi_x+ \psi\|^2 + b \|\psi_x\|^2 \leq c\eps\bigg[\frac{1}{|\lambda|}+1\bigg]\|U\|_\H^2 +
\frac{c}{\eps^3}\bigg[\frac{1}{|\lambda|^2}+1\bigg]\|U\|_\H\|\widehat U\|_\H.
$$
holds for some structural constant
$c>0$ independent of $\lambda$ and $\eps$.
\end{Lemma}

\begin{proof}
For the b.c.\ \eqref{BCM-beg} the result follows immediately from Lemma \ref{lemma-varphi_x+psi}
(recall that in this situation $\mathcal{P}(\varphi,\psi,\alpha)
=\mathcal{Q}(\psi,\beta)=0$). Hence, we only need to treat the b.c.\ \eqref{BCD-beg}.
By Lemma \ref{lemma-term-front}, we have
\begin{align*}
\mathcal{P}(\varphi,\psi,\alpha)
+ \mathcal{Q}(\psi,\beta) &\leq
c\eps\|U\|^2_\H + c\eps\big[\|\theta_x\|
+\|\xi_x\|\big]\|U\|_\H\\
&\quad + \dfrac{c}{\varepsilon}|\lambda|\big[\|\eta\|_{\cal M}+
\|\zeta\|_{\cal N}\big]\|U\|_\H
+\dfrac{c}{\varepsilon}\|U\|_\H \|\widehat{U}\|_\H.
\end{align*}
In the light of Lemmas \ref{ETA+ZETA} and \ref{Theta-aux}, the right-hand
side above is less than or equal to
$$
c\eps\|U\|^2_\H + c\eps\big[1+|\lambda|\big] \|U\|_\H \sqrt{\|U\|_\H \|\widehat{U}\|_\H}
+ \dfrac{c}{\varepsilon}|\lambda|\|U\|_\H\sqrt{\|U\|_\H \|\widehat{U}\|_\H}
+\dfrac{c}{\varepsilon}\|U\|_\H \|\widehat{U}\|_\H
$$
for every $\eps\in(0,1)$, where
$c>0$ is independent of $\lambda$ and $\eps$.
Applying Lemma \ref{lemma-varphi_x+psi}, we arrive at
\begin{align*}
k\|\varphi_x+ \psi\|^2 + b \|\psi_x\|^2 &\leq c\eps\bigg[\frac{1}{|\lambda|}+1\bigg]\|U\|_\H^2
+\frac{c}{\eps}\bigg[\frac{1}{|\lambda|}+1\bigg]\|U\|_\H \sqrt{\|U\|_\H \|\widehat{U}\|_\H}\\
&\quad +\frac{c}{\eps}\bigg[\frac{1}{|\lambda|^2}+1\bigg]\|U\|_\H\|\widehat U\|_\H\\
&\leq c\eps\bigg[\frac{1}{|\lambda|}+1\bigg]\|U\|_\H^2 +
\frac{c}{\eps^3}\bigg[\frac{1}{|\lambda|^2}+1\bigg]\|U\|_\H\|\widehat U\|_\H
\end{align*}
for every $\lambda\neq0$ and every $\eps\in(0,1)$. The thesis has been proved.
\end{proof}

\subsection{Proof of Theorem \ref{EXP-STAB-TEO} (completion)}
In the light of the Gearhart-Pr\"{u}ss-Huang theorem \cite{gearhart,huang,pruss}
(see also \cite{Zhuangyi.Liu.book}),
the conclusion of Theorem \ref{EXP-STAB-TEO} follows provided that $\i\R$
is contained into the resolvent set $\rho(\cal{A})$ of $\cal{A}$
and
\begin{equation}
\label{GPcond}
\limsup_{|\lambda|\to\infty}\|(\i\lambda  - \mathcal{A})^{-1}\|_{\mathcal{L}(\mathcal{H})}<\infty.
\end{equation}
In the next two propositions we verify these conditions.

\begin{prop}\label{ir-resolv}
The inclusion $\i\mathbb{R}\subset\rho(\cal{A})$ holds.
\end{prop}

\begin{proof}
Let us assume by contradiction that $\i\lambda_0\notin \rho(\cal{A})$ for some $\lambda_0\in \mathbb{R}$.
Since $\mathcal{A}$ generates a contraction semigroup, then $\i\lambda_0$ is
necessarily an approximate eigenvalue (see e.g.\ \cite[Proposition B.2]{Arent-Batty-Hieber}).
This amounts to saying that there exists
$U_n=(\varphi_n, \Phi_n, \psi_n, \Phi_n, \theta_n,\eta_n, \xi_n, \zeta_n)\in \D(\mathcal{A})$ satisfying
\begin{equation}\label{ir-resolv-eq1}
\|U_n\|_\H=1 \qquad \mbox{and} \qquad \i\lambda_0 U_n - \mathcal{A}U_n
= \widehat{U}_n\rightarrow 0 \;\, \mbox{in}\;\, \H.
\end{equation}
We limit ourself to consider the case
$\lambda_0\ne 0$.
When $\lambda_0=0$ the argument
can be carried out arguing similarly as in the proof of \cite[Theorem 7.10]{dellobressetim}
and the details are left to the reader.

Using Lemmas \ref{ETA+ZETA}, \ref{THETA+XI}, \ref{lemma-Phi}
and \ref{corol-varphix-psi-final}, we estimate
\begin{align*}
\|U_n\|_\H^2 &\leq c\eps\bigg[\frac{1}{|\lambda_0|}+1\bigg]\|U_n\|_\H^2
+ c\|U_n\|_\H\big[\|\varphi_{nx}+\psi_n\|+\|\psi_{nx}\|\big]
+\frac{c}{\eps^3}\bigg[\frac{1}{|\lambda_0|^2}+1\bigg]\|U_n\|_\H\|\widehat U_n\|_\H\\\noalign{\vskip0.7mm}
&\leq c\eps\bigg[\frac{1}{|\lambda_0|}+1\bigg]\|U_n\|_\H^2
+\frac{c}{\eps^7}\bigg[\frac{1}{|\lambda_0|^2}+1\bigg]\|U_n\|_\H\|\widehat U_n\|_\H,
\end{align*}
for every $\varepsilon \in (0,1)$ and some $c>0$ independent of $\eps$ and $\lambda_0$.
Fixing now $\varepsilon= \varepsilon(\lambda_0)\in(0,1)$ small enough that
\begin{equation*}
\varepsilon< \dfrac{1}{2c}\bigg[\frac{1}{|\lambda_0|}+1\bigg]^{-1},
\end{equation*}	
there exists a constant $K= K(\lambda_0)> 0$ such that $
\|U_n\|_{\H}\leq K\|\widehat{U}_n\|_{\H}$, contradicting
\eqref{ir-resolv-eq1}.
\end{proof}

\begin{prop}
Condition \eqref{GPcond} holds.
\end{prop}

\begin{proof}
It is sufficient to show that, for every $|\lambda|>1$, the inequality
\begin{equation}
\label{boundexp}
\|U\|_\H\leq c \|\widehat{U}\|_\H
\end{equation}
holds for some structural constant $c>0$ independent of $\lambda$. Once this bound
has been established, Proposition \ref{ir-resolv} ensures that
$
\|(\i\lambda - \mathcal{A})^{-1}\|_{\mathcal{L}(\mathcal{H})}\leq c
$
for every $|\lambda|>1$, and the latter yields \eqref{GPcond}.

In order to prove \eqref{boundexp} we first notice that
for $|\lambda|>1$ the conclusion of Lemma \ref{corol-varphix-psi-final}
becomes
$$
k\|\varphi_x+ \psi\|^2 + b \|\psi_x\|^2 \leq c\varepsilon\|U\|_\H^2
+\frac{c}{\eps^3}\|U\|_\H\|\widehat U\|_\H.
$$
Combining the estimate above with Lemmas \ref{ETA+ZETA}, \ref{THETA+XI} and \ref{lemma-Phi},
for every $\varepsilon\in (0,1)$ we have
\begin{align*}
\|U\|_\H^2 \leq c\eps\|U\|_\H^2
+c\|U\|_\H\big[\|\varphi_x+\psi\|+\|\psi_x\|\big]+\dfrac{c}{\varepsilon^3}\|U\|_\H \|\widehat{U}\|_\H
\leq c\eps\|U\|_\H^2 + \dfrac{c}{\varepsilon^7}\|U\|_\H \|\widehat{U}\|_\H
\end{align*}
with $c>0$ independent of $\lambda$ and $\eps$.
Fixing $\varepsilon\in (0,1)$ small enough that $c\eps\leq\frac12$, we reach \eqref{boundexp}.
\end{proof}

\section{Concluding Remarks}
\label{finsec}

\medskip
\noindent
{\bf I.}  As already mentioned in the Introduction, the Gurtin-Pipkin law is
more general than the Cattaneo
one. Indeed, choosing for $\tau,\varsigma>0$
\begin{equation}
\label{choice}
g(s)=g_\tau(s) = \frac{1}{\tau}\e^{-\frac{s}{\tau}}\qquad \text{and} \qquad
h(s)= h_\varsigma(s)=\frac{1}{\varsigma}\e^{-\frac{s}{\varsigma}},
\end{equation}
and defining the heat-flux variables
\begin{align*}
&q(x,t) = - \varpi_1 \int_0^\infty g_\tau(s) \theta_x(x,t-s) \d s,\\\noalign{\vskip0.7mm}
&p(x,t) = - \varpi_2 \int_0^\infty h_\varsigma(s) \xi_x(x,t-s) \d s,
\end{align*}
by means of an elementary calculation one can see that $q$ and $p$ satisfy \eqref{catta}
provided that $\theta$ and $\xi$ satisfy \eqref{GP}. This correspondence is not merely formal
and indeed, arguing as in \cite[Section 8]{delloro-pata},
it is possible to show rigorously that the semigroup $S(t)$
generated by the Timoshenko-Gurtin-Pipkin system
corresponding to the particular choice \eqref{choice} is
exponentially stable if and only if the same does the semigroup generated by the
Timoshenko-Cattaneo system.

\medskip
\noindent
{\bf II.} Also the Fourier law can be recovered from the Gurtin-Pipkin one by means of a singular
limit procedure. To see that, we consider for $\eps>0$ the rescaled kernels
$$
g_\eps (s) = \frac{1}{\eps} g \left(\frac{s}{\eps} \right) \qquad \text{and}\qquad
h_\eps (s) = \frac{1}{\eps} h \left(\frac{s}{\eps} \right)
$$
which converge in the distributional sense to the Dirac mass $\delta_0$ as $\eps\to0$. In this way,
system~\eqref{GP} boils down to \eqref{sysfou} in the (singular) limit $\eps\to0$
(see \cite{AMNESIA,delloro-pata} for more details).
Using a similar procedure
it is also possible to recover the Timoshenko-Coleman-Gurtin system, which consists
in replacing \eqref{GP} with the equations (see \cite{CGu})
\begin{equation}\label{COLGUR}
\left\{\begin{array}{lcl}
\displaystyle\rho_3 \theta_{t} - \varpi_1 (1-\ell)\theta_{xx}
-\varpi_1 \ell \int_{0}^{\infty} g(s) \theta_{xx}(t-s) \mathrm{d} s
+ \gamma (\varphi_{x} + \psi)_t = 0,\medskip\\
\displaystyle	\rho_4 \xi_{t} - \varpi_2(1-\ell)\xi_{xx} -
\varpi_2 \ell\int_{0}^{\infty} h(s) \xi_{xx}(t-s) \mathrm{d} s + \sigma \psi_{xt} = 0.
\end{array} \right.
\end{equation}
Here, $\ell \in (0,1)$ is a fixed parameter and
the limit cases $\ell=0,1$ correspond to the Fourier and the Gurtin-Pipkin models,
respectively. Considering for $\eps>0$ the rescaled kernels
$$
g_{\eps} (s) = \frac{1-\ell}{\eps} g \left(\frac{s}{\eps} \right) + \ell g(s)\qquad \text{and} \qquad
h_{\eps} (s) = \frac{1-\ell}{\eps} h \left(\frac{s}{\eps} \right) + \ell h(s),
$$
we have the convergence
$g_\eps\to (1-\ell)\delta_0 + \ell g$ and $h_\eps\to(1-\ell)\delta_0 + \ell h$
in the distributional sense for $\eps\to0$, and thus
system~\eqref{GP} with the choice $g=g_\eps$ and $h=h_\eps$ boils down to \eqref{COLGUR}.

\medskip
\noindent
{\bf III.} Although so far we have assumed
that the temperatures fulfill the same constitutive law, it is
possible to analyze Timoshenko systems where $\theta$ and $\xi$ obey
different laws.
In order to illustrate all the possible cases that can be covered,
let us use the following abbreviations:
\begin{itemize}
	\item Gurtin-Pipkin {\bf (GP)};
	\item Fourier {\bf (F)};
	\item  Cattaneo {\bf (C)};
	\item Coleman-Gurtin {\bf (CG)}.
\end{itemize}
All the models listed in the table below are either a particular instance of
system \eqref{tim4-beg}
or else can be recovered from it by means of appropriate singular limit procedures.
The corresponding solutions semigroups are exponentially stable
independently of the values of the structural parameters.

\begin{table}[h]

 \small
\centering
\begin{tabular}{c|c }
Coupling on shear force ($\theta$) &  Coupling on bending moment ($\xi$)\\
\hline
  {\bf (GP)} &  {\bf (GP)} (our problem)
        \\
  {\bf (GP)} & {\bf (F)}  \\
  {\bf (GP)} & {\bf (C)}            \\
  {\bf (GP)} & {\bf (CG)}          \\ \hline
	{\bf (F)} &  {\bf (GP)} \\
	{\bf (F)}&  {\bf (F)}   (problem in \cite{alvesetal-JEE}) \\
	{\bf (F)} &  {\bf (C)} \\
	{\bf (F)} &  {\bf (CG)} \\ \hline
	{\bf (C)} &  {\bf (GP)} \\
{\bf (C)} & {\bf (F)} \\
{\bf (C)} &  {\bf (C)} (problem in \cite{DJ})\\
{\bf (C)} & {\bf (CG)}  \\ \hline
	{\bf (CG)} &  {\bf (GP)} \\
{\bf (CG)} &  {\bf (F)}  \\
{\bf (CG)} & {\bf (C)} \\
{\bf (CG)} & {\bf (CG)}   	
	\end{tabular}
\end{table}

\smallskip
\noindent
{\bf IV.} Finally, we mention that the analysis carried out in this work
can be adapted also to different boundary conditions.
For instance, one can assume the
mixed Neumann-Dirichlet boundary conditions considered in \cite{DJ}
$$
\varphi_x(0,t) = \varphi_x(L,t) = \psi(0,t) = \psi(L,t) = \theta(0,t) = \theta(L,t) = \xi(0,t) = \xi(L,t)=0,
$$
or any of the boundary conditions considered in \cite{alvesetal-JEE}.
Clearly, appropriate modifications and precise computations must be done, but no substantial challenges arise.


\end{document}